\documentclass[11pt,a4paper,reqno]{amsart}

\usepackage[T1]{fontenc}
\usepackage[utf8]{inputenc}
\usepackage{lmodern}
\usepackage{microtype}
\usepackage{amsmath,amssymb,amsthm}
\usepackage[margin=27mm]{geometry}
\usepackage[hidelinks]{hyperref}
\newtheorem{theorem}{Theorem}[section]
\newtheorem{lemma}[theorem]{Lemma}
\newtheorem{corollary}[theorem]{Corollary}

\newcommand{\N}{\mathbb N}
\newcommand{\Qpos}{\mathbb Q_{>0}}
\newcommand{\T}{\mathbb T}

\title[Standard morphisms and Pythagorean triples]{Standard morphisms and Pythagorean triples}
\author{Jo\~ao Ara\'ujo}
\address{Departamento de Matem\'atica, NOVA School of Science and Technology, Universidade NOVA de Lisboa, 2829-516 Caparica, Portugal}
\email{joao.araujo@fct.unl.pt}
\author{Andr\'e Carvalho}
\address{Centro de Investigação em Matemática e Aplicações (CIMA), Departamento de Matem\'atica, Escola de Ci\^encias e Tecnologia, Universidade de \'Evora, \'Evora, Portugal}
\email{andre.carvalho@uevora.pt}
\date{}
\subjclass[2020]{Primary 05D10; Secondary 11D09, 11N64}
\keywords{Pythagorean triples, standard morphisms, completely multiplicative functions, unavoidability threshold}

\begin{document}

\begin{abstract}
Let $m\geq 1$, let $f:\mathbb N\to\mathbb Z/m\mathbb Z$ be a standard morphism and let $T(m)$ be the least integer $N$ such that every such $f$ admits a primitive monochromatic Pythagorean triple with hypotenuse at most $N$.   The aim of this note is to prove that every standard morphism has infinitely many identity-valued Pythagorean triples and infinitely many primitive monochromatic Pythagorean triples.  Thus the qualitative part of Problem~4.3 of Eliahou, Fromentin, Marion-Poty and Robilliard is solved for every $m$.  Moreover, $T(m)$ is finite, the morphism $n\mapsto v_3(n)\pmod m$ has least possible hypotenuse $(9^m+1)/2$ and $T(d)\leq T(m)$ when $d\mid m$.
\end{abstract}

\maketitle

\section{The qualitative problem}

Throughout this note $\N=\{1,2,\ldots\}$ and maps will act on the right.  A map
\[
 f:\N\longrightarrow\mathbb Z/m\mathbb Z
\]
is said to be a \emph{standard morphism} if
\[
 (ab)f=af+bf,\qquad a,b\in\N.
\]
Thus a standard morphism is determined freely by its values on the primes.

 Eliahou, Fromentin, Marion-Poty and Robilliard proved that the least
 integers $N$ such that every standard morphism admits a monochromatic
 Pythagorean triple with hypotenuse at most $N$ are $533$ and $4633$
 for $m=2$ and $m=3$, respectively \cite[Propositions~4.1 and~4.2]{EFMR}.  They asked in \cite[Problem~4.3]{EFMR} whether every standard morphism to $\mathbb Z/m\mathbb Z$, for $m\geq4$, admits a monochromatic Pythagorean triple and how the corresponding threshold grows with $m$.

We write $\mathbb T=\{z\in\mathbb C:|z|=1\}$ for the circle group. Frantzikinakis, Klurman and Moreira proved the following theorem \cite[Theorem~1.5]{FKM}.

\begin{theorem}\label{thm:FKM}
Let $h:\N\to\T$ be completely multiplicative and finite-valued.  Then there are distinct $x,y,z\in\N$ such that
\[
 x^2+y^2=z^2
 \quad\hbox{and}\quad
 xh=yh=zh=1.
\]
\end{theorem}

We call a pair $(u,v)\in\mathbb Q_{>0}^2$ a \emph{Pythagorean pair} if \( u^2+v^2=1. \)
The aim of this section is to derive the qualitative part of Problem~4.3 and the infinitude statements needed below.  We start with a separation lemma.

\begin{lemma}\label{lem:separation}
Let $q_1,\ldots,q_s\in\Qpos^{\times}\setminus\{1\}$ and let $m\geq1$.  There are a prime $\ell\nmid m$ and a homomorphism
\[
 \lambda:\Qpos^{\times}\longrightarrow\mathbb Z/\ell\mathbb Z
\]
such that $q_i\lambda\neq0$, for all $i\in\{1,\ldots,s\}$.
\end{lemma}

\begin{proof}
Let $S$ be the set of primes occurring in the numerators and denominators of $q_1,\ldots,q_s$.  Under the valuation isomorphism the subgroup generated by $S$ is $\mathbb Z^S$.  By $v_i$ we will denote the non-zero vector corresponding to $q_i$.

Choose a prime $\ell>s$ such that $\ell\nmid m$ and no $v_i$ is zero modulo $\ell$.  In the dual space $(\mathbb F_\ell^S)^*$, the functionals vanishing on $v_i$ form a hyperplane, say $H_i$.  Therefore
\[
 \left|\bigcup_{i=1}^s H_i\right|
 \leq s\ell^{|S|-1}
 <\ell^{|S|}.
\]
Thus there is a linear functional which is non-zero on every $v_i$.  Extend it by zero on the valuations at the remaining primes.  The resulting homomorphism is the required $\lambda$.  The lemma is proved.
\end{proof}

We now prove the main result.

\begin{theorem}\label{thm:main}
Let $m\geq1$ and let $f:\N\to\mathbb Z/m\mathbb Z$ be a standard morphism.  Then:
\begin{enumerate}
\item there are infinitely many Pythagorean triples $(x,y,z)$ such that $xf=yf=zf=0$;
\item there are infinitely many primitive Pythagorean triples $(a,b,c)$ such that $af=bf=cf$.
\end{enumerate}
\end{theorem}

\begin{proof}
Extend $f$ to the homomorphism
\[
 \widehat f:\Qpos^{\times}\longrightarrow\mathbb Z/m\mathbb Z,
 \qquad
 (a/b)\widehat f=af-bf,
\]
and put $H=\ker\widehat f$.  The map
\[
 nh=e^{\frac{2\pi i(nf)}{m}}
\]
is completely multiplicative and finite-valued.  By Theorem~\ref{thm:FKM}, there is an identity-valued Pythagorean triple.  Dividing its two legs by its hypotenuse, it follows that $H$ contains a Pythagorean pair.

We claim that $H$ contains infinitely many Pythagorean pairs.  In fact, suppose that its Pythagorean pairs are
\[
 (u_1,v_1),\ldots,(u_s,v_s).
\]
Every $u_i$ belongs to $(0,1)$ and hence $u_i\neq1$.  Apply the previous lemma to $u_1,\ldots,u_s$.  Since $\ell\nmid m$, the group
\[
 \mathbb Z/m\mathbb Z\times\mathbb Z/\ell\mathbb Z
\]
is cyclic.  Thus the product homomorphism $(\widehat f,\lambda)$ embeds in $\T$.  Applying Theorem~\ref{thm:FKM} to its restriction to $\N$, we obtain a Pythagorean pair $(u,v)$ in $H\cap\ker\lambda$.  This pair is equal to one of the pairs chosen above.  Therefore $u_i\lambda=0$, for some $i$, a contradiction with the choice of $\lambda$.  It is proved that $H$ contains infinitely many Pythagorean pairs.

Let $(u,v)$ be one of these pairs.  There is a unique primitive Pythagorean triple $(a,b,c)$ such that
\[
 (u,v)=(a/c,b/c).
\]
Since $u,v\in H$, we have
\[
 af=cf
 \quad\hbox{and}\quad
 bf=cf.
\]
Apart from interchanging the two legs, distinct pairs give distinct primitive triples.  Hence (2) follows.

Now let $af=bf=cf=\alpha$ and put $t=c^{m-1}$ when $m>1$.  Then
\[
 (ta)f=(tb)f=(tc)f=m\alpha=0.
\]
When $m=1$ take $t=1$.  The rational pair associated with $(ta,tb,tc)$ is still $(u,v)$.  The infinitely many pairs constructed above therefore give infinitely many distinct identity-valued triples, and hence (1) is proved.  The theorem is proved.
\end{proof}

The following consequence will be useful when standard morphisms are viewed as finite quotients of $\Qpos^{\times}$.

\begin{corollary}\label{cor:cosets}
Let $H\leq\Qpos^{\times}$ and suppose that $\Qpos^{\times}/H$ is finite cyclic.  Every coset of $H$ contains infinitely many Pythagorean triples entrywise.
\end{corollary}

\begin{proof}
Let $\varphi:\Qpos^{\times}\to\Qpos^{\times}/H$ be the quotient map.  The restriction of $\varphi$ to $\N$ is a standard morphism after identifying the quotient with a cyclic group.  By the previous theorem, $H$ contains infinitely many integral Pythagorean triples.

Let $e$ be the exponent of the quotient and let $qH$ be a coset.  Write $q=a/b$, where $a,b\in\N$, and put $t=ab^{e-1}$.  Since $t/q=b^e\in H$, we have $t\in qH$.  Multiplying every triple in $H$ by $t$ gives infinitely many triples whose three entries belong to $qH$.  The corollary is proved.
\end{proof}

\section{The unavoidability threshold}

For $m\geq1$, by $T(m)$ we will denote the least integer $N$ such that every standard morphism $f:\N\to\mathbb Z/m\mathbb Z$ has a primitive monochromatic Pythagorean triple with hypotenuse at most $N$.  Put $T(m)=\infty$ if there is no such integer.  Dividing a Pythagorean triple by the common divisor of its entries subtracts the same value of a standard morphism from the three colours.  Thus allowing non-primitive Pythagorean triples gives the same threshold.

The aim of this section is to prove the following theorem.

\begin{theorem}\label{thm:threshold}
Let $m\geq1$.  Then:
\begin{enumerate}
\item $T(m)<\infty$;
\item for the standard morphism $nf_m=v_3(n)\pmod m$, where $v_3(n)$ denotes the $3$-adic valuation of $n$, the least possible hypotenuse is $(9^m+1)/2$;
\item $T(d)\leq T(m)$ whenever $d\mid m$.
\end{enumerate}
\end{theorem}

\begin{proof}
A standard morphism is determined by its values on the primes. Let $\mathbb P$ denote the set of primes.  Consider the compact product space
\[
 X_m=(\mathbb Z/m\mathbb Z)^{\mathcal P}.
\]
For every primitive Pythagorean triple $P$, let $U_P$ be the set of assignments for which $P$ is monochromatic.  The set $U_P$ depends on finitely many prime coordinates and hence is open and closed.  By Theorem~\ref{thm:main}, the sets $U_P$ cover $X_m$.  A finite number of them cover $X_m$ as well.  Taking the maximum of their hypotenuses, it follows that $T(m)<\infty$.  Thus (1) is proved.

We now prove (2).  The quadratic residues modulo $3$ are $0$ and $1$.  In a primitive Pythagorean triple the hypotenuse is not divisible by $3$, since otherwise both legs would be divisible by $3$.  Hence exactly one leg is divisible by $3$.  Therefore monochromaticity for $f_m$ forces the divisible leg to have positive $3$-adic valuation congruent to $0$ modulo $m$.  In particular, this leg is divisible by $3^m$.

Write the primitive triple as
\[
 \{a,b\}=\{r^2-s^2,2rs\},
 \qquad c=r^2+s^2,
\]
where $r>s>0$ are coprime and have opposite parity.  Suppose first that $3^m\mid r^2-s^2$.  The odd integers $r-s$ and $r+s$ are coprime.  Thus one of them is divisible by $3^m$, and hence
\[
 2c=(r-s)^2+(r+s)^2\geq3^{2m}+1.
\]
Suppose now that $3^m\mid2rs$.  As $r$ and $s$ are coprime, one of them is divisible by $3^m$, and hence $c\geq3^{2m}+1$.  In both cases
\[
 c\geq\frac{9^m+1}{2}.
\]

Take now
\[
 r=\frac{3^m+1}{2}
 \quad\hbox{and}\quad
 s=\frac{3^m-1}{2}.
\]
These integers are coprime and have opposite parity.  They give the primitive triple
\[
 \left(3^m,\frac{3^{2m}-1}{2},\frac{3^{2m}+1}{2}\right).
\]
The three entries have the same image under $f_m$.  It is proved that the least possible hypotenuse is $(9^m+1)/2$, and hence (2) follows.

Finally, suppose that $d\mid m$.  The map
\[
 \iota:\mathbb Z/d\mathbb Z\longrightarrow\mathbb Z/m\mathbb Z,
 \qquad
 (a+d\mathbb Z)\iota=\frac md a+m\mathbb Z,
\]
is injective.  If $f:\N\to\mathbb Z/d\mathbb Z$ is a standard morphism, then $f\iota$ is a standard morphism and a triple is monochromatic for $f\iota$ if and only if it is monochromatic for $f$.  Thus $T(d)\leq T(m)$, and hence (3) is proved.  The theorem is proved.
\end{proof}

It follows from the previous theorem that
\[
 \frac{9^m+1}{2}\leq T(m)<\infty.
\]
The exact values $T(2)=533$ and $T(3)=4633$ show that the valuation obstruction is not sharp in these two cases.  The determination of an effective upper bound and of the asymptotic growth of $T(m)$ remains open.

\section*{Acknowledgements}
The first author is supported by national funds through the FCT – Fundação para a Ciência e a Tecnologia, I.P., under the projects UID/297/2025 and UID/PRR/297/2025 (Center for Mathematics and Applications - NOVA Math).

The second author was supported by national funds through the Fundação para a Ciência e Tecnologia, FCT, under the project
UID/04674/2025.

\end{document}